\documentclass[reqno,12pt]{amsart}

\usepackage{amssymb}

\numberwithin{equation}{section}

\newcommand{\Z}{{\mathbb Z}}

\newcommand{\C}{{\mathbb C}}

\newcommand{\h}{{\mathfrak h}}
\newcommand{\wh}{{\widehat{\mathfrak h}}}

\newcommand{\wnu}{{\widehat{\nu}}}

\newcommand{\al}{\alpha}

\newcommand{\la}{\langle}
\newcommand{\ra}{\rangle}

\DeclareMathOperator{\Aut}{Aut}

\newtheorem{thm}{Theorem}[section]
\newtheorem{prop}[thm]{Proposition}
\newtheorem{lem}[thm]{Lemma}

\theoremstyle{definition}

\newtheorem{notation}[thm]{Notation}

\theoremstyle{remark}
\newtheorem{rmk}[thm]{Remark}

\newcommand{\vir}{\mathrm{Vir}}

\renewcommand{\hom}{\mathrm{Hom}}

\begin{document}

\title[$\Z_p$-orbifold constructions]
{A remark on $\Z_p$-orbifold constructions of the Moonshine vertex operator algebra}

\author[T. Abe]{Toshiyuki Abe}
\address{Faculty of Education, 
Ehime University, Matsuyama, Ehime 790-8577, Japan}
\email{abe.toshiyuki.mz@ehime-u.ac.jp}

\author[C.H. Lam]{Ching Hung Lam}
\address{Institute of Mathematics, Academia Sinica, Taipei 115, Taiwan}
\email{chlam@math.sinica.edu.tw}

\author[H. Yamada]{Hiromichi Yamada}
\address{Department of Mathematics, Hitotsubashi University, Kunitachi,
Tokyo 186-8601, Japan, 
Institute of Mathematics, Academia Sinica, Taipei 115, Taiwan}
\email{yamada.h@r.hit-u.ac.jp}

\thanks{T.A. is partially supported by JSPS fellow 15K04823. 
C.L. is partially supported by MoST grant 104-2115-M-001-004-MY3 of Taiwan. }

\subjclass[2010]{17B69, 17B65}

\keywords{orbifold construction, Leech lattice, Moonshine vertex operator algebra}

\begin{abstract}
For $p = 3,5,7,13$, we consider a $\Z_p$-orbifold construction of 
the Moonshine vertex operator algebra $V^\natural$. 
We show that the vertex operator algebra obtained by the $\Z_p$-orbifold construction 
on the Leech lattice vertex operator algebra $V_\Lambda$ and a lift of 
a fixed-point-free isometry of order $p$ is isomorphic to the Moonshine vertex operator
algebra $V^\natural$. 
We also describe the relationship between those $\Z_p$-orbifold constructions 
and the $\Z_2$-orbifold construction in a uniform manner. 
In Appendix, we give a characterization of the Moonshine 
vertex operator algebra 
$V^\natural$ 
by two mutually orthogonal Ising vectors.
\end{abstract}

\maketitle

\section{Introduction}

Let $V$ be a vertex operator algebra (VOA) and $G \subset \Aut V$ a finite group 
of automorphisms of $V$. 
The fixed-point subalgebra $V^G = \{ v \in V | g v = v, g  \in G\}$ is 
called an orbifold subVOA. 
A simple current extension of an orbifold subVOA provides an effective method to
construct a new VOA from a given pair of a VOA 
and an automorphism of the VOA of finite order. 

The first example of such a simple current extension of an orbifold subVOA is the Moonshine 
VOA $V^\natural$ constructed in \cite{FLM1988}, where the VOA 
$V_\Lambda$ associated with the Leech lattice $\Lambda$ and an automorphism $\theta$ 
of order $2$ lifted from the $-1$ isometry of $\Lambda$ are considered. 
Using the construction in \cite{FLM1988}, many maximal $2$-local subgroups of 
the Monster are described as the stabilizers of some subVOAs of 
$V_\Lambda^{\la \theta \ra}$ in \cite{Sh07}. 
However, it is not easy to describe $p$-local subgroups based on this construction 
for $p \ne 2$.

For each $p = 3, 5, 7, 13$, there is a unique, up to conjugacy, 
fixed-point-free isometry of the Leech lattice $\Lambda$ of order $p$. 
In the introduction of \cite{FLM1988}, a similar construction of $V^\natural$ by using 
an automorphism $\tau$ of $V_\Lambda$ of order $p$ lifted from a fixed-point-free 
isometry of $\Lambda$ of order $p$ was conjectured (see also \cite{DMZp}). 
Such a construction of $V^\natural$ was obtained in \cite{CLS2017} for the case $p = 3$ 
(see also \cite{Miyamoto2013, TY2013}). 
Moreover, certain maximal $3$-local subgroups of the Monster are described relatively 
explicitly using the $\Z_3$-orbifold construction.   

In this paper, we consider the $\Z_p$-orbifold construction of the Moonshine VOA 
$V^\natural$ for the cases $p = 3, 5, 7, 13$. 
We show that the VOA obtained by the $\Z_p$-orbifold construction on $V_\Lambda$ 
and  a lift of a fixed-point-free isometry of $\Lambda$ of order $p$ is isomorphic to the Moonshine VOA 
$V^\natural$. 
We also describe the relationship between those $\Z_p$-orbifold constructions 
and the $\Z_2$-orbifold construction \cite{FLM1988} in a uniform manner. 

Our idea is to use an isometry of the Leech lattice $\Lambda$ of order $2p$ whose $i$th power  
is fixed-point-free on $\Lambda$ for all $1 \le i \le 2p-1$. 
Such an isometry is unique up to conjugacy. 
The isometry can be lifted to an automorphism $\sigma$ of $V_\Lambda$ 
of the same order $2p$ with $\theta = \sigma^p$ and $\tau  = \sigma^{p+1}$. 
We consider the fixed-point subVOA  
$V_\Lambda^{\la \sigma \ra} = \{ v \in V_\Lambda | \sigma v = v\}$ by $\sigma$.
It is known that $V_\Lambda^{\la \sigma \ra}$ has exactly $4p^2$ irreducible modules 
up to equivalence \cite{CM2016, Miyamoto2015, MT2004}, 
and all of them are simple current modules \cite{DRX2015, EMS2015, Moeller2016}. 
These irreducible $V_\Lambda^{\la \sigma \ra}$-modules are parametrized 
by an abelian group $D \cong \Z_{2p} \times \Z_{2p}$ as $W^\alpha$, $\alpha \in D$ so that 
the fusion products are given by $W^\alpha \boxtimes W^\beta \cong W^{\alpha + \beta}$ 
\cite{EMS2015, Moeller2016}. 

It turns out that the VOA obtained by applying a $\Z_{2p}$-orbifold construction to 
the Leech lattice VOA $V_\Lambda$ and the automorphism $\sigma$ is isomorphic to 
$V_\Lambda$, also. 
This isomorphism induces some explicit relations between those $\Z_p$-orbifold 
constructions and the $\Z_2$-orbifold construction \cite{FLM1988}.  
In fact, it follows from \cite{EMS2015} that there are four subgroups $H_r$, $1 \le r \le 4$, 
of $D$ for which a 
simple current extension 
$\widetilde{V}^{(r)} = \bigoplus_{\alpha \in H_r} W^\alpha$ 
has a structure of a simple, rational, $C_2$-cofinite, holomorphic VOA  
of CFT-type which extends the VOA structure of $V_\Lambda^{\la \sigma \ra}$.  
Among these four 
VOAs $\widetilde{V}^{(r)}$, $1 \le r\le 4$, 
two are isomorphic to the Leech lattice VOA $V_\Lambda$ and 
the other two are isomorphic to the Moonshine VOA $V^\natural$. 
As for the latter two cases, 
one is a decomposition of the $\Z_p$-orbifold construction of $V^\natural$ 
by the automorphism $\tau \in \Aut V_\Lambda$ of order $p$ into a direct sum of 
irreducible $V_\Lambda^{\la \sigma \ra}$-modules, 
and the other is a decomposition of the $\Z_2$-orbifold construction of $V^\natural$ 
by the involution $\theta$ obtained in \cite{FLM1988}  into a direct sum of 
irreducible $V_\Lambda^{\la \sigma \ra}$-modules. 

The paper is organized as follows. 
In Section \ref{sec:preliminaries}, 
we recall some results on cyclic orbifold constructions, 
irreducible twisted modules for lattice vertex operator algebras, 
and certain fixed-point-free isometries of the Leech lattice of order $2p$, 
$p = 3,5,7,13$. 
In Section \ref{sec:irred_sigma_i-twisted}, we discuss some basic properties 
of the irreducible $\sigma^i$-twisted $V_\Lambda$-modules. 
Section \ref{sec:Z_2p-orbifold_construction} is devoted to $\Z_{2p}$-orbifold 
constructions by the automorphism $\sigma$. 
In Appendix, we give a characterization of the Moonshine VOA $V^\natural$ 
by two mutually orthogonal Ising vectors.

\section{Preliminaries}\label{sec:preliminaries}
In this section, we recall some results on cyclic orbifold constructions 
\cite{EMS2015, Moeller2016}, 
irreducible twisted modules for VOAs associated with positive definite even lattices 
and isometries of finite order \cite{BK2004, DL1996, Lepowsky1985}, 
and some fixed-point-free isometries of the Leech lattice of order $2p$, 
$p = 3,5,7,13$ \cite{ATLAS}.

\subsection{Cyclic orbifold constructions}
\label{subsec:cyclic_orbifold_construction}

We follow the notation in \cite{EMS2015}. 
Let $V$ be a simple, rational, $C_2$-cofinite, holomorphic vertex operator algebra 
of CFT-type and $G = \langle g \rangle$ a cyclic group of automorphisms of $V$ 
of order $n$. 
Then there is a unique irreducible $h$-twisted $V$-module $V(h)$ 
for $h \in G$ by \cite{DLM2000}. 

For each $h \in G$, there is a projective representation $\phi_h$ of $G$ 
on the vector space $V(h)$ such that 
\begin{equation*}
\phi_h(g) Y_{V(h)}(v,z) \phi_h(g)^{-1} = Y_{V(h)}(gv,z)
\end{equation*}
for $v \in V$. 
The representation $\phi_h$ is unique up to multiplication by an $n$th root of unity. 
If $h = 1$, we have $V(h) = V$ and then assume $\phi_h(g) = g$. 
We write $\phi_i$ for $\phi_{g^i}$. 

Let $W^{(i,j)}$ be the eigenspace of $\phi_i(g)$ in $V(g^i)$ with eigenvalue 
$e^{2\pi\sqrt{-1}j/n}$, i.e., 
\begin{equation*}
W^{(i,j)} = \{ w \in V(g^i) | \phi_i(g) w = e^{2\pi\sqrt{-1}j/n} w\}.
\end{equation*}
Then $W^{(i,j)}$ is an irreducible $V^G$-module and 
\begin{equation*}
V(g^i) = \bigoplus_{j=0}^{n-1} W^{(i,j)}
\end{equation*}
is an eigenspace decomposition of $V(g^i)$ for $\phi_i(g)$. 
The indices $i$ and $j$ are considered to be modulo $n$.
The second index $j$ depends on the choice of multiplication by an $n$th root of unity 
for the representation $\phi_i$ if $i \ne 0$. 
Note that $W^{(0,0)} = V^G$. 

The irreducible $V^G$-modules $W^{(i,j)}$, $i, j \in \Z_n$, form a complete set of 
representatives of equivalence classes of irreducible $V^G$-modules 
\cite{CM2016, Miyamoto2015, MT2004}, 
and all of them are simple current modules \cite{DRX2015, EMS2015}. 

The conformal weight of $V(g^i)$, $i \in \Z_n$ plays an important role in \cite{EMS2015}.
In the special case where the conformal weight of $V(g)$ belongs to $(1/n)\Z$, 
the fusion algebra of the orbifold subVOA $V^G$ of $V$ by $G$ has particularly nice form. 
We summarize the results of \cite[Section 5]{EMS2015} 
as the following theorem for later use.

\begin{thm}\label{thm:EMS2015} 
$($\cite{EMS2015}$)$ 
Let $V$ and $G = \langle g \rangle$ be as above. 
If the conformal weight of $V(g)$ belongs to $(1/n)\Z$, then we can choose 
multiplication of $\phi_i$ by an $n$th root of unity so that 
the following conditions hold.

$(1)$ $W^{(i,j)} \boxtimes W^{(k,l)} \cong W^{(i+k, j+l)}$.

$(2)$ The conformal weight of $W^{(i,j)}$ is $q_{\Delta}((i,j)) \equiv ij/n \pmod{\Z}$.

$(3)$ The fusion algebra of $V^G$ is the group algebra of $\Z_n \times \Z_n$ 
with a quadratic form $q_{\Delta}$.

$(4)$ Let $H$ be an isotropic subgroup of $\Z_n \times \Z_n$ 
with respect to the quadratic form $q_{\Delta}$. 
Then 
\begin{equation*}
\bigoplus_{(i,j) \in H} W^{(i,j)}
\end{equation*}
admits a structure of a simple, rational, $C_2$-cofinite, 
self-contragredient VOA  
of CFT-type which extends the VOA structure of $V^G$. 
Furthermore, if $H$ is a maximal isotropic subgroup, then it is holomorphic.
\end{thm}

The subgroup $\{ (i,0) | i \in \Z_n \}$ is always a maximal isotropic subgroup of 
$\Z_n \times \Z_n$ and 
\begin{equation*}
\widetilde{V}_g = \bigoplus_{i \in \Z_n} W^{(i,0)}
\end{equation*}
is a simple, rational, $C_2$-cofinite, holomorphic VOA of CFT-type 
which extends the VOA structure of $V^G$ \cite[page 21]{EMS2015}. 
We say that the VOA $\widetilde{V}_g$ is obtained by a $\Z_n$-orbifold construction 
for $V$ and $g$.

\subsection{Irreducible twisted modules for $V_L$}
\label{subsec:irred_twisted_V_L-module}

Irreducible twisted modules for a lattice VOA $V_L$ with respect to 
a lift of an isometry $\nu$ of $L$ of finite order were constructed explicitly in 
\cite{BK2004, DL1996, Lepowsky1985}. 
In this section, we recall some basic properties of those irreducible twisted modules 
in the special case where $L$ is unimodular and $\nu$ is fixed-point-free.

Let $(L, \langle \cdot , \cdot \rangle)$ be a positive definite even unimodular lattice and 
$\nu$ a fixed-point-free isometry of $L$ of finite order. 
Let $m$ be a positive integer such that $\nu^m = 1$. 
Note that $m$ is not necessarily the order of $\nu$. 
We extend the isometry $\nu$ to $\h = L \otimes_\Z \C$ linearly. 
Following \cite[(4.17)]{BK2004} and \cite[Remark 3.1]{DLM2000}, 
let
\begin{equation*}
\h^{(i; \nu)} = \{ h \in \h | \nu h = \xi_m^{-i} h \}, \quad \xi_m = e^{2\pi\sqrt{-1}/m}.
\end{equation*}

Since $\nu$ is fixed-point-free, we have $\h^{(0; \nu)} = 0$ and 
$\h = \bigoplus_{i=1}^{m-1} \h^{(i; \nu)}$. 
The $\nu$-twisted affine Lie algebra $\wh[\nu]$ is defined by 
\begin{equation*}
\wh[\nu] 
= \Big( \bigoplus_{i=1}^{m-1} \h^{(i; \nu)} \otimes t^{i/m} \C[t,t^{-1}] \Big) \oplus \C K
\end{equation*}
with commutation relations
\begin{equation*}
[x \otimes t^n, y \otimes t^{n'}] = \langle x, y\rangle n \delta_{n+n', 0}K, 
\quad [K, \wh[\nu]] = 0
\end{equation*}
for $x \in \h^{(i; \nu)}$, $y \in \h^{(i'; \nu)}$, $n \in i/m + \Z$, $n' \in i'/m + \Z$. 
We write $x(n)$ for $x \otimes t^n$. 

The index $i$ of $\h^{(i; \nu)}$ can be considered to be modulo $m$. 
Then $\wh[\nu]$ is also denoted as
\begin{equation*}
\wh[\nu] = \Big( \bigoplus_{n \in (1/m)\Z} \h^{(mn; \nu)} \otimes \C t^n \Big) \oplus \C K.
\end{equation*}

Let $\wnu$ be an automorphism of the VOA $V_L$ which is a lift of $\nu$. 
Since $L$ is unimodular, there is a unique irreducible $\wnu$-twisted $V_L$-module 
up to equivalence \cite{DLM2000}. 
The irreducible $\wnu$-twisted $V_L$-module constructed in \cite{DL1996, Lepowsky1985} 
is of the form
\begin{equation*}
V_L(\wnu) = M(1)[\nu] \otimes T
\end{equation*} 
as a vector space, 
where $M(1)[\nu]$ is the symmetric algebra $S(\wh[\nu]^-)$ of an abelian Lie algebra 
\begin{equation*}
\wh[\nu]^- = \bigoplus_{n \in (1/m)\Z_{< 0}} \h^{(mn; \nu)} \otimes \C t^n
\end{equation*}
and $T$ is an irreducible module for a certain finite group. 
The symmetric algebra $S(\wh[\nu]^-)$ is spanned by the elements of the form 
\begin{equation*}
h_r(-n_r) \cdots h_1(-n_1) 1
\end{equation*}
with $r \in \Z_{\ge 0}$, $n_j \in (1/m)\Z_{> 0}$ and $h_j \in \h^{(-mn_j;\nu)}$, $1 \le j \le r$.
The weight of an element $h_r(-n_r) \cdots h_1(-n_1) 1 \otimes u \in V_L(\wnu)$ 
with $u \in T$ is given by 
\begin{equation}\label{eq:wt_of_irred_twisted_module}
n_1 + \cdots + n_r + \rho,
\end{equation}  
where
\begin{equation}\label{eq:conformal_wt_formula_1}
\rho = \rho(V_L(\widehat{\nu})) = \frac{1}{4 m^2} \sum_{i = 1}^{m-1} i (m-i) \dim \h^{(i;\nu)}
\end{equation}
is the conformal weight of $V_L(\widehat{\nu})$. 

The dimension of $T$ is determined in \cite[(4.53)]{BK2004},  
\cite[Proposition 6.2]{Lepowsky1985} and we have 
\begin{equation}\label{eq.dim_T_general}
\dim T = \sqrt{ |L/(1-\nu)L|}.
\end{equation}

\subsection{Fixed-point-free isometries of $\Lambda$ of order $2p$, $p = 3,5,7,13$}
\label{subsec:sigma}

The automorphism group $Co_0 = O(\Lambda)$ of the Leech lattice $\Lambda$ 
is a central extension of the largest Conway group $Co_1$ by a group of order $2$. 
The central element of $O(\Lambda)$ of order $2$ is the $-1$ isometry 
$\theta : \alpha \mapsto -\alpha$ for $\alpha \in \Lambda$. 
The character of the natural representation of $O(\Lambda)$ on the $24$ dimensional space 
$\Lambda \otimes_\Z \C$ is denoted by $\chi_{102}$ in \cite[page 186]{ATLAS}. 
We see from the values of $\chi_{102}$ that the following lemma holds.

\begin{lem}\label{lem:def-of-sigma}
For $p = 3, 5, 7, 13$, there exists a unique, up to conjugacy, isometry $\tau \in O(\Lambda)$ 
of order $p$ which acts fixed-point-freely on $\Lambda$. 
Let $\sigma = \theta \tau \in O(\Lambda)$. 
Then $\sigma$ is of order $2p$ and $\sigma^i$ acts fixed-point-freely on $\Lambda$ 
for all $1 \le i \le 2p-1$.
\end{lem}

\begin{rmk}
For $p = 3, 5, 7, 13$, the isometry $\sigma$ of Lemma \ref{lem:def-of-sigma} is 
the unique, up to conjugacy, isometry of $\Lambda$ of order $2p$ such that $\sigma^i$ 
acts  fixed-point-freely on $\Lambda$ for all $1 \le i \le 2p-1$.
\end{rmk}

\section{Irreducible $\sigma^i$-twisted $V_\Lambda$-modules}
\label{sec:irred_sigma_i-twisted}

For $p = 3, 5, 7, 13$, let $\sigma$ be as in Section \ref{subsec:sigma}. 
Thus $\sigma$ is an isometry of the Leech lattice 
$\Lambda$ of order $m = 2p$ and 
$\sigma^i$ is fixed-point-free on $\Lambda$ for all $1 \le i \le m-1$.
Since $\sigma^{m/2}$ is the $-1$ isometry of $\Lambda$, we have 
$\la \al, \sigma^{m/2} \al \ra = - \la \al, \al \ra \in 2\Z$ for $\al \in \Lambda$. 
Moreover, $\Lambda^\sigma = \{ \al \in \Lambda | \sigma \al = \al \} = 0$. 
Hence there is a lift $\widehat{\sigma} \in \Aut V_\Lambda$ of $\sigma$ of order $m$ 
by \cite[Proposition 7.2]{EMS2015}. 
For simplicity of notation, we use the  same symbol $\sigma$ to denote 
$\widehat{\sigma}$. 
That is, $\sigma$ denotes both an isometry of $\Lambda$ 
of order $m$ whose $i$th power is fixed-point-free on $\Lambda$ for all $1 \le i \le m-1$  
and an automorphism of the VOA  
$V_\Lambda$ of order $m$ which is a lift of the isometry.

Let 
\begin{equation*}
\theta = \sigma^p, \quad \tau = \sigma^{p+1}.
\end{equation*}
Then $\sigma  = \theta \tau = \tau \theta$, $|\theta| = 2$, $|\tau| = p$, 
$\la \sigma  \ra = \la \theta, \tau \ra$ 
and $\theta$ is a lift of the $-1$ isometry of $\Lambda$.

We follow the notation in Section \ref{subsec:irred_twisted_V_L-module} 
with $L = \Lambda$ and $\nu = \sigma^i$ for the irreducible 
$\sigma^i$-twisted $V_\Lambda$-modules 
$V_\Lambda(\sigma^i) = M(1)[\sigma^i] \otimes T$, 
$1 \le i \le m-1$. 
Thus $\h = \Lambda \otimes_\Z \C$ and 
\begin{equation}\label{eq:sigma_i-eigenspace}
\h^{(j; \sigma^i)} = \{ h \in \h | \sigma^i h = \xi_m^{-j} h \}, \quad \xi_m = e^{2\pi\sqrt{-1}/m}.
\end{equation}

\subsection{Conformal weight of $V_\Lambda(\sigma^i)$}

The conformal weight of $V_\Lambda(\sigma^i)$ is 
\begin{equation}\label{eq:conformal_wt_formula_2}
\rho(V_\Lambda(\sigma^i)) = \frac{1}{4 m^2} \sum_{j = 1}^{m-1} j (m-j) \dim \h^{(j;\sigma^i)}.
\end{equation}
by Eq. \eqref{eq:conformal_wt_formula_1}.

\begin{lem}\label{lem:dim_of_h_eigenspace}
The dimension of $\h^{(j; \sigma^i)}$, $1 \le i \le m-1$, $j \in \Z_m$ is as follows.

$(1)$ If $i$ is odd and $i \ne p$, then
\begin{equation*}
\dim \h^{(j; \sigma^i)} = 
\begin{cases}
24/(p-1) & (j \text{ is odd}, j \ne p),\\
0 & (\text{otherwise}).
\end{cases}
\end{equation*}

$(2)$ If $i$ is even, then
\begin{equation*}
\dim \h^{(j; \sigma^i)} = 
\begin{cases}
24/(p-1) & (j \text{ is even}, j \ne 0),\\
0 & (\text{otherwise}).
\end{cases}
\end{equation*}

$(3)$ If $i = p$, then
\begin{equation*}
\dim \h^{(j; \sigma^p)} = 
\begin{cases}
24 & (j = p),\\
0 & (j \ne p).
\end{cases}
\end{equation*}
\end{lem}

\begin{proof}
Let $h \in \h^{(j; \sigma^i)}$. 
First, assume that $i$ is odd and $i \ne p$. 
If $j$ is even, then $jp \equiv 0 \pmod{m}$ and 
$(\sigma^i)^p h = \xi_m^{-jp} h = h$. 
Since $(\sigma^i)^p = \theta$, this implies that $h = 0$. 
If $j = p$, then $\sigma^i h = \xi_m^{-p}h = -h$ and $\sigma^{i-p} h = h$. 
Since $\sigma^k$ is fixed-point-free on $\Lambda$ for $0 \ne k \in \Z_m$, 
it follows that $h = 0$. 
Therefore, we have the eigenspace decomposition 
\begin{equation*}
\h = \bigoplus_{\substack{1 \le k \le p\\ k \ne (p+1)/2}} \h^{(2k-1; \sigma^i)}
\end{equation*}
of $\h$ by $\sigma^i$.
Since $m$ and $2k-1$ are coprime for $1 \le k \le p$, $k \ne (p+1)/2$, we have
\begin{equation*}
\dim \h^{(2k-1; \sigma^i)} = \dim \h^{(2k-1; (\sigma^i)^{2k-1})}
= \dim \h^{(1; \sigma^i)}.
\end{equation*}
Thus the assertion (1) holds.

Next, assume that $i$ is even. 
Then the order of $\sigma^i$ is the prime $p$. 
The eigenspace decomposition of $\h$ by $\sigma^i$ is
\begin{equation*}
\h = \bigoplus_{1 \le k \le p-1} \h^{(2k; \sigma^i)}
\end{equation*}
and $\dim h^{(2k; \sigma^i)}$, $1 \le k \le p-1$ coincide each other.
Hence the assertion (2) holds.

Since $\sigma^p$ is the $-1$ isometry of $\Lambda$, the assertion (3) is clear.
\end{proof}

By Lemma \ref{lem:dim_of_h_eigenspace} and Eq. \eqref{eq:conformal_wt_formula_2}, 
we can calculate the conformal weight of $V_\Lambda(\sigma^i)$.
\begin{lem}\label{lem:conformal_wt}
The conformal weight of $V_\Lambda(\sigma^i)$, $1 \le i \le m-1$ is as follows.
\begin{equation*}
\rho(V_\Lambda(\sigma^i)) = 
\begin{cases}
(2p-1)/2p & (i \text{ is odd}, i \ne p),\\
(p+1)/p & (i \text{ is even}),\\
3/2 & (i = p).
\end{cases}
\end{equation*}
\end{lem}

\subsection{Dimension of $T$}

The dimension of $T$ of the irreducible $\sigma^i$-twisted $V_\Lambda$-module
$V_\Lambda(\sigma^i) = M(1)[\sigma^i] \otimes T$ is 
\begin{equation}\label{eq.dim_T_sigma_i}
\dim T = \sqrt{ |\Lambda/(1-\sigma^i)\Lambda|}.
\end{equation}
by \eqref{eq.dim_T_general}. 

\begin{lem}\label{lem:dim_T}
The dimension of $T$, $1 \le i \le m-1$ is as follows.
\begin{equation*}
\dim T = 
\begin{cases}
1 & (i \text{ is odd}, i \ne p),\\
p^{12/(p-1)} & (i \text{ is even}),\\
2^{12} & (i = p).
\end{cases}
\end{equation*}
\end{lem}

\begin{proof}
If $i$ is odd and $i \ne p$, 
then the eigenvalues of $\sigma^i$ on $\h$ are exactly the primitive $m$th 
roots of unity by Lemma \ref{lem:dim_of_h_eigenspace} (1), and so 
the minimal polynomial of $\sigma^i$ on $\h$ is a cyclotomic polynomial
\begin{equation*}
\begin{split}
\Phi_m(x) 
&= \sum_{k=0}^{p-1} (-1)^k x^k\\
&= (x-1) \big( \sum_{k=1}^{(p-1)/2} x^{2k-1} \big) + 1.
\end{split}
\end{equation*}

Hence 
\begin{equation*}
\al = (1 - \sigma^i) \big( \sum_{k=1}^{(p-1)/2} \sigma^{i(2k-1)} \big) \al 
\in (1 - \sigma^i) \Lambda
\end{equation*}
for $\al \in \Lambda$. 
Thus $(1 - \sigma^i) \Lambda = \Lambda$ and $\dim T = 1$. 

If $i$ is even, then the order of $\sigma^i$ is $p$ 
and the minimal polynomial of $\sigma^i$ on $\h$ is a cyclotomic polynomial
\begin{equation*}
\Phi_p(x) = x^{p-1} + \cdots + x + 1
\end{equation*}
by Lemma \ref{lem:dim_of_h_eigenspace} (2). 
Hence $\Lambda/(1-\sigma^i)\Lambda \cong \Z_p^{24/(p-1)}$ 
by \cite[Lemma A.1]{GL2011} and $\dim T = p^{12/(p-1)}$ . 
If $i = p$, then $\Lambda/(1-\sigma^i)\Lambda  = \Lambda/2\Lambda \cong \Z_2^{24}$ 
and $\dim T = 2^{12}$.
\end{proof}

\section{$\Z_{2p}$-orbifold constructions}
\label{sec:Z_2p-orbifold_construction}

We keep the notation in Section \ref{sec:irred_sigma_i-twisted}. 
Since the conformal weight of the irreducible $\sigma$-twisted $V_\Lambda$-module 
$V_\Lambda(\sigma)$ belongs to $(1/m)\Z$ by Lemma \ref{lem:conformal_wt}, 
for each $i \in \Z_m$,  
we can choose a representation $\phi_i$ of the group $\la \sigma \ra$ on 
the irreducible $\sigma^i$-twisted $V_\Lambda$-module $V_\Lambda(\sigma^i)$ 
so that the eigenspace $W^{(i,j)}$ of $\phi_i(\sigma)$ in $V_\Lambda(\sigma^i)$, 
$i,j \in \Z_m$ satisfy the four conditions 
of Theorem \ref{thm:EMS2015} 
with $V = V_\Lambda$, $g = \sigma$ and $n = m = 2p$.
In particular, 

$(1)$ $W^{(i,j)} \boxtimes W^{(k,l)} \cong W^{(i+k, j+l)}$.

$(2)$ The conformal weight of $W^{(i,j)}$ is $q_{\Delta}((i,j)) \equiv ij/m \pmod{\Z}$.

The eigenspace decomposition of $V_\Lambda(\sigma^i)$ for $\phi_i(\sigma)$ is 
\begin{equation}\label{eq:sigma_i-twisted_dec}
V_\Lambda(\sigma^i) = \bigoplus_{j \in \Z_m} W^{(i,j)}
\end{equation}

The condition (2) implies that 
only $W^{(2k-1,0)}$ is of integral weight among $W^{(2k-1,j)}$, $j \in \Z_m$ 
for $i= 2k-1$, $1 \le k \le p$, $k \ne (p+1)/2$;  
only $W^{(2k,0)}$ and $W^{(2k,p)}$ are of integral weight among 
$W^{(2k,j)}$, $j \in \Z_m$ for $i = 2k$, $0 \le k \le p-1$; and 
$W^{(p,j)}$ is of integral weight if and only if $j$ is even for $i = p$.

There are four maximal isotropic subgroups of $\Z_m \times \Z_m$ 
with respect to the quadratic form $q_{\Delta}$, namely,
\begin{align*}
H_1 &= \{(0,j) | j \in \Z_m\}, & H_2&= \{ (i,0) | i \in \Z_m\},\\
H_3 &= \{ (2k,pk) | k \in \Z_m\}, & H_4 &= \{ (pk,2k) | k \in \Z_m\}.
\end{align*}

For $r = 1,2,3,4$, 
\begin{equation*}
\widetilde{V}^{(r)} = \bigoplus_{(i,j) \in H_r} W^{(i,j)}
\end{equation*}
admits a structure of a simple, rational, $C_2$-cofinite, holomorphic VOA  
of CFT-type which extends the VOA structure of 
$V_\Lambda^{\la \sigma \ra}$ 
by Theorem \ref{thm:EMS2015} (4).

In the case $r=1$, 
\begin{equation*}
\widetilde{V}^{(1)} = \bigoplus_{j=0}^{2p-1} W^{(0,j)}
\end{equation*}
is the eigenspace decomposition of $V_\Lambda$ by the automorphism $\sigma$, 
that is, $\widetilde{V}^{(1)} = V_\Lambda$. Moreover,
\begin{equation}
V_\Lambda^+ = \bigoplus_{k=0}^{p-1} W^{(0,2k)}, \qquad 
V_\Lambda^- = \bigoplus_{k=0}^{p-1} W^{(0,2k+1)},
\end{equation}
where 
\begin{equation*}
V_\Lambda^{\pm} = \{ v \in V_\Lambda | \theta v = \pm v\}.
\end{equation*}

Next, we consider 
\begin{equation*}
\widetilde{V}^{(2)} = \bigoplus_{i=0}^{2p-1} W^{(i,0)}.
\end{equation*}

We calculate the weight $1$ subspace $(\widetilde{V}^{(2)})_1$ of $\widetilde{V}^{(2)}$. 
If $i$ is even or $i = p$, then the conformal weight of $V_\Lambda(\sigma^i)$ 
is greater than $1$ by Lemma \ref{lem:conformal_wt}. 
Hence 
\begin{equation*}
(\widetilde{V}^{(2)})_1 = \bigoplus_{\substack{0 \le k \le p-1\\ k \ne (p-1)/2}} 
(W^{(2k+1,0)})_1.
\end{equation*}

For each $0 \le k \le p-1$, $k \ne (p-1)/2$, the conformal weight of 
$V_\Lambda(\sigma^{2k+1})$ is $(2p-1)/2p$ by Lemma \ref{lem:conformal_wt} 
and $W^{(2k+1,0)}$ is the only irreducible $V_\Lambda^{\la \sigma \ra}$-module 
with integral weights among $W^{(2k+1,j)}$, $j \in \Z_m$ 
by the condition (2) of Theorem \ref{thm:EMS2015}.
Thus by Eq. \eqref{eq:wt_of_irred_twisted_module} and 
Lemmas \ref{lem:dim_of_h_eigenspace} and \ref{lem:dim_T}, we see that 
\begin{equation*}
\dim (W^{(2k+1,0)})_1 = 24/(p-1).
\end{equation*}

Therefore, $\dim (\widetilde{V}^{(2)})_1 = 24$. 
Furthermore, $a_{(0)} b =  0$ for $a, b \in (\widetilde{V}^{(2)})_1$ by the fusion rule  
$W^{(i,0)} \boxtimes W^{(j,0)} \cong W^{(i+j, 0)}$ as $i + j$ is even if both $i$ and $j$ 
are odd. 
Then $\widetilde{V}^{(2)}$ is a holomorphic VOA of central charge $24$ 
whose weight $1$ subspace is a $24$ dimensional abelian Lie algebra with respect to 
the bracket $[a,b] = a_{(0)} b$. 
Thus $\widetilde{V}^{(2)} = \bigoplus_{i=0}^{2p-1} W^{(i,0)}$ is isomorphic to 
$V_\Lambda$ by \cite[Theorem 3]{DMb} and we obtain the following theorem.

\begin{thm}\label{thm:V-2}
$V_\Lambda \cong \bigoplus_{i=0}^{2p-1} W^{(i,0)}$.
\end{thm}

Define a linear isomorphism $\theta' : \widetilde{V}^{(2)} \to \widetilde{V}^{(2)}$  
by $(-1)^i$ on $W^{(i,0)}$. Then $\theta'$ is an automorphism of the VOA 
$\widetilde{V}^{(2)}$ by the fusion rule 
$W^{(i,0)} \boxtimes W^{(j,0)} \cong W^{(i+j,0)}$. 
The automorphism $\theta'$ is of order $2$ and $-1$ on the weight $1$ subspace. 
Hence it is conjugate to a lift $\theta$ of the $-1$ isometry of $\Lambda$.
Indeed, $\theta^{-1}\theta'$ acts as the identity on the weight $1$ subspace of 
$\widetilde{V}^{(2)} = V_\Lambda$.
Hence $\theta^{-1}\theta'$ is an inner automorphism $e^{h(0)}$ for some $h\in \h$ by 
\cite[Lemma 2.5]{DN1999} and we have
\begin{equation*}
\theta' = \theta e^{h(0)} = e^{-\frac{1}{2}h(0)} \theta e^{\frac{1}{2}h(0)}
\end{equation*} 
as required. Then
\begin{equation*}
V_\Lambda^+ = V_\Lambda^{\la \theta \ra} 
\cong (\widetilde{V}^{(2)})^{\la \theta' \ra} = \bigoplus_{k=0}^{p-1} W^{(2k,0)}.
\end{equation*} 

Therefore,  the following theorem holds.
\begin{thm}\label{thm:v3}
$V_\Lambda^+ \cong \bigoplus_{k=0}^{p-1} W^{(2k,0)}$ and 
$V_\Lambda^- \cong \bigoplus_{k=0}^{p-1} W^{(2k+1,0)}$.
\end{thm}

Next, we consider $\widetilde{V}^{(3)} = \bigoplus_{k=0}^{2p-1} W^{(2k,pk)}$. 
Note that $2(p+i)\equiv 2i\pmod{2p}$. Moreover, $pi\equiv 0 \pmod{2p}$ if $i$ is even and 
$pi\equiv p \pmod{2p}$ if $i$ is odd. 
Hence 
\begin{equation*}
H_3 = \{ (2k,0) | 0 \le k \le p-1\} \cup \{ (2k,p) | 0 \le k \le p-1\}.
\end{equation*}
Thus
$\widetilde{V}^{(3)}$ contains $\bigoplus_{k=0}^{p-1} W^{(2k,0)} \cong V_\Lambda^+$. 
The VOA $V_\Lambda^+$ has exactly four irreducible modules, 
namely, $V_\Lambda^\pm$ and $V_\Lambda^{T,\pm}$, where 
$V_\Lambda^{T,\pm}$ are the eigenspaces with eigenvalues $\pm 1$ of $\theta$ 
in a unique irreducible $\theta$-twisted $V_\Lambda$-module $V_\Lambda(\theta)$. 
All of those four irreducible $V_\Lambda^+$-modules are simple current. 
We take $V_\Lambda^{T,\pm}$ 
so that the conformal weight of 
$V_\Lambda^{T,+}$ is $2$ and that of $V_\Lambda^{T,-}$ is $3/2$. 
The weight $1$ subspace of $\widetilde{V}^{(3)}$ is $0$, 
for the conformal weight of $V_\Lambda(\sigma^{2k})$, $1 \le k \le p-1$ 
is greater than $1$ by Lemma \ref{lem:conformal_wt}. 
Hence we conclude that $\widetilde{V}^{(3)}$ is isomorphic to the Moonshine VOA  
$V^\natural = V_\Lambda^+ \oplus V_\Lambda^{T,+} $ constructed in \cite{FLM1988}. 
Thus the following theorem holds. 
\begin{thm}\label{eq:Z_2p-orbifold_construction}
$V^\natural
\cong \bigoplus_{k=0}^{2p-1} W^{(2k,pk)}$ with 
$V_\Lambda^+ \cong \bigoplus_{k=0}^{p-1} W^{(2k,0)}$
and 
$V_\Lambda^{T,+} \cong \bigoplus_{k=0}^{p-1} W^{(2k,p)}$.
\end{thm}

Recall that $\tau = \sigma^{p+1}$ is of order $p$.  
We also note that 
\begin{equation}\label{eq:H_3_another_form}
\begin{split}
H_3 
&= \{ ((p+1)k, pk) | 0 \le k \le 2p-1 \}\\
&= \{ ((p+1)k, 0) | 0 \le k \le p-1 \} \cup \{ ((p+1)k, p) | 0 \le k \le p-1 \}.
\end{split}
\end{equation}

Let 
\begin{equation}\label{eq:U-decomposition}
U^{(k,0)} = W^{((p+1)k,0)} \oplus W^{((p+1)k,p)}, \quad 0 \le k \le p-1.
\end{equation}

For $v \in V_\Lambda$, we have $\tau v = v$ if and only if $v$ is a sum of eigenvectors 
of  $\sigma$ of eigenvalues $\pm 1$. 
Hence $U^{(0,0)} = W^{(0,0)} \oplus W^{(0,p)}$ is equal to 
\begin{equation}
V_\Lambda ^{\la \tau \ra} = \{ v \in V_\Lambda | \tau v = v \}.
\end{equation}

The irreducible $\tau^k = \sigma^{(p+1)k}$-twisted $V_\Lambda$-module 
$V_\Lambda(\tau^k) = V_\Lambda(\sigma^{(p+1)k})$, 
$1 \le k \le p-1$  
is a direct sum of 
eigenspaces for $\tau$. 
For each $k$, there is a unique one with integral weights among those 
eigenspaces for $\tau$, namely, $U^{(k,0)}$. 
Since the conformal weight of $V_\Lambda(\tau) = V_\Lambda(\sigma^{p+1})$ 
belongs to $(1/p)\Z$ by Lemma \ref{lem:conformal_wt},
\begin{equation}\label{eq:tau-Z_p-orbifold_construction}
\widetilde{V}_{\Lambda,\tau} = \bigoplus_{k=0}^{p-1} U^{(k,0)}
\end{equation}
is a $\Z_p$-orbifold construction with respect to $V_\Lambda$ and $\tau$ 
by Theorem \ref{thm:EMS2015}. 

Both $V^\natural$ and $\widetilde{V}_{\Lambda,\tau}$ are simple current extensions 
of $V_\Lambda^{\la \sigma \ra}$ with the same simple current components 
$W^{(i,j)}$, $(i,j) \in H_3$ 
by Theorem \ref{eq:Z_2p-orbifold_construction}, 
Eq. \eqref{eq:H_3_another_form}, \eqref{eq:U-decomposition} 
and \eqref{eq:tau-Z_p-orbifold_construction}. 
Since the VOA structure of a simple current extension 
is unique (see \cite{DMq} and \cite[Proposition 5.3]{DM2004}), 
it follows that $V^\natural$ and $\widetilde{V}_{\Lambda,\tau}$ are 
isomorphic as VOAs. 

\begin{thm}\label{thm:main}
$\widetilde{V}_{\Lambda,\tau} \cong V^\natural$.
\end{thm}

\begin{rmk}
We have $V_\Lambda^+ = \bigoplus_{k=0}^{p-1} W^{(0,2k)}$ and 
$V_\Lambda^{T,+} = \bigoplus_{k=0}^{p-1} W^{(p,2k)}$, 
for $V_\Lambda(\theta) = V_\Lambda(\sigma^p)$ and 
$W^{(p,j)}$ is of integral weight only if $j$ is even. 
Since 
\begin{equation*}
H_4 = \{ (0,2k) | 0 \le k \le p-1\} \cup \{ (p,2k) | 0 \le k \le p-1\},
\end{equation*}
$\widetilde{V}^{(4)} =  \bigoplus_{k=0}^{2p-1} W^{(pk,2k)}$ is a decomposition of 
the Moonshine VOA  
$V^\natural = V_\Lambda^+ \oplus V_\Lambda^{T,+}$ into a direct sum of 
irreducible $V_\Lambda^{\la \sigma \ra}$-modules.
\end{rmk}

\appendix
\section{A characterization of the Moonshine VOA}

In this appendix, we give another characterization of the Moonshine VOA $V^\natural$ 
using Ising vectors. 
It also provides an alternative proof that 
$\widetilde{V}_{\Lambda,\tau} \cong V^\natural$ 
for $p=3$ and $5$. 
The main theorem is as follows.
\begin{thm}
Let $V$ be a simple, rational, $C_2$-cofinite, holomorphic VOA of CFT-type with 
central charge $24$ such that the weight $1$ subspace $V_1 = 0$. 
If $V$ contains two mutually orthogonal Ising vectors, 
then $V$ is isomorphic to the Moonshine VOA $V^\natural$.
\end{thm}

The idea is essentially the same as in \cite{LY1}  
and is also similar to that in Section \ref{sec:Z_2p-orbifold_construction}. 
We try to obtain  the Leech lattice VOA $V_\Lambda$ by using some
$\Z_2$-orbifold construction on $V$. 
The extra assumption on Ising vectors is used to 
define an involution on $V$ such that the corresponding twisted module 
has conformal weight one.

An element $e\in V_2$ 
is called an \textit{Ising vector} if the vertex subalgebra $\vir(e)$ generated by $e$ is 
isomorphic to the simple Virasoro VOA $L(1/2,0)$ of central charge $1/2$. 
Let $V_e(h)$ be the sum of all irreducible $\vir(e)$-submodules of $V$ isomorphic
to $L(1/2,h)$ for $h=0,1/2,1/16$. 
Then one has an isotypical decomposition:
\begin{equation*}
V=V_e(0)\oplus V_e(\frac{1}2)\oplus V_e(\frac{1}{16}).
\end{equation*}

Recall from \cite{M1} that the linear automorphism $\tau_e$ which acts as $1$ on 
$V_e(0)\oplus V_e(\frac{1}2)$ and $-1$ on $V_e(\frac{1}{16})$ defines an automorphism 
of the VOA $V$. 
On the fixed point subVOA $V^{\la \tau_e\ra}=V_e(0)\oplus V_e(\frac{1}2)$, 
the linear automorphism $\sigma_e$ which acts as $1$ on $V_e(0)$ 
and $-1$ on $V_e(\frac{1}{2})$ also defines an automorphism on $V^{\la \tau_e\ra}$.  

Let $e$ and $f$ be two mutually orthogonal Ising vectors in $V$ and let $U$
be the subVOA generated by $e$ and $f$. Then
\begin{equation*}
U = \vir(e) \otimes \vir(f) \cong L(\frac{1}2,0)\otimes L(\frac{1}2,0).
\end{equation*}

For any $h_1, h_2\in \{0, 1/2, 1/16\}$, we define the space of
multiplicities of the irreducible $U$-module $L(\frac{1}2, h_1) \otimes L(\frac{1}2,h_2)$ 
in $V$ by
\begin{equation*}
M(h_1, h_2)=\hom_{U}(L(\frac{1}2, h_1) \otimes L(\frac{1}2,h_2),V).
\end{equation*}
Then we have the isotypical decomposition
\begin{equation*}
V= \bigoplus_{h_1, h_2\in \{ 0, \frac{1}2, \frac{1}{16}\}} L(\frac{1}2, h_1) 
\otimes L(\frac{1}2,h_2)\otimes  M(h_1, h_2).
\end{equation*}

Notice that $M(0,0) =\mathrm{Com}_V(U) = U^c$  
is a subVOA of central charge $23$ and $M(h_1, h_2)$, $h_1, h_2\in \{0, 1/2, 1/16\}$, are
$M(0,0)$-modules.
Note also  that  $$(V^{\la \tau_e, \tau_f\ra})^{\la \sigma_e, \sigma_f\ra} 
=  L(\frac{1}2, 0) \otimes L(\frac{1}2,0)\otimes  M(0, 0).$$

\begin{prop}
The subVOA $M(0,0)$ is $C_2$-cofinite and rational.
Moreover, $M(h_1, h_2)$, $h_1, h_2\in \{0, 1/2\}$, are simple current modules of $M(0,0)$.
\end{prop}

\begin{proof}
Let $E=\la \tau_e, \tau_f\ra \subset \Aut V$ 
be the subgroup generated by the Miyamoto involutions $\tau_e$ and $\tau_f$.  
Then $E$ is elementary abelian of order $4$  and the fixed point subVOA is 
\[
V^E = \bigoplus_{h_1, h_2\in \{0, \frac{1}2\}} L(\frac{1}2, h_1) 
\otimes L(\frac{1}2,h_2)\otimes  M(h_1, h_2).
\]
Then  $V^E$ is $C_2$-cofinite and rational by a result of Carnahan and Miyamoto 
\cite{CM2016, Miyamoto2015}. 

Let $S=\la \sigma_e, \sigma_f\ra \subset \Aut V^E$. Then $S$ is also 
elementary abelian and hence the fixed point subVOA
\[
W= (V^E)^S = L(\frac{1}2, 0) \otimes L(\frac{1}2,0)\otimes  M(0, 0)
\]
is $C_2$-cofinite and rational.  Therefore,  $M(0, 0)$ is also $C_2$-cofinite and rational.

That $M(h_1, h_2)$, $h_1, h_2\in \{0, 1/2\}$, are simple current modules of $M(0,0)$ 
follows from the fact that $L(\frac{1}2, h_1) \otimes L(\frac{1}2,h_2)\otimes  M(h_1, h_2)$ 
are common eigenspaces of $S$ on $V^E$ and \cite[Remark 6.4]{DJX}.
\end{proof}

\begin{notation}
For $i,j\in \{0,1\}$, let $V^{(i,j)} = \{v\in V\mid \tau_e v=(-1)^iv, \tau_f v=(-1)^jv\}$.
Then
\begin{equation*}
\begin{split}
V^{(0,0)}=& V^E = \bigoplus_{h_1, h_2\in \{0, \frac{1}2\}} L(\frac{1}2, h_1) 
\otimes L(\frac{1}2,h_2)\otimes  M(h_1, h_2),\\
V^{(1,0)}=& L(\frac{1}2, \frac{1}{16}) \otimes L(\frac{1}2,0)\otimes  M(\frac{1}{16}, 0) \oplus
L(\frac{1}2, \frac{1}{16}) \otimes L(\frac{1}2,\frac{1}2)\otimes  M(\frac{1}{16}, \frac{1}2), \\
V^{(0,1)}=& L(\frac{1}2,0)\otimes L(\frac{1}2, \frac{1}{16}) \otimes   M(0, \frac{1}{16}) \oplus
L(\frac{1}2,\frac{1}2)\otimes  L(\frac{1}2, \frac{1}{16}) \otimes M(\frac{1}2,\frac{1}{16}), \\
V^{(1,1)}=& L(\frac{1}2, \frac{1}{16}) \otimes L(\frac{1}2, \frac{1}{16}) 
\otimes M(\frac{1}{16}, \frac{1}{16}).
\end{split}
\end{equation*}

Notice that $V^{(i,j)}, i,j\in \{0,1\},$ are simple current modules of $V^E$ \cite{DJX}.
\end{notation}

\begin{lem}Let $V$ be a 
simple, rational, $C_2$-cofinite, holomorphic VOA 
of CFT type with central charge $24$ such that $V_1=0$. 
Then $M(h_1, h_2)\neq 0$ for any $h_1, h_2\in \{0,\frac{1}2, \frac{1}{16}\}$.  
\end{lem}

\begin{proof}
Recall $U = \vir(e) \otimes \vir(f)\cong L(\frac{1}2, 0) \otimes L(\frac{1}2,0)$.  
Then the double commutant $(U^c)^c$ is an extension of $U$. 
Note that there is only one non-trivial extension of $U$, 
which is isomorphic to 
$L(\frac{1}2, 0) \otimes L(\frac{1}2,0) 
\oplus L(\frac{1}2, \frac{1}2) \otimes L(\frac{1}2,\frac{1}2)$ 
and the weight one subspace is non-zero. 
Hence $(U^c)^c=U$, for $V_1=0$.  
Therefore, by a result of Krauel and Miyamoto \cite{KMi}, all irreducible modules of $U$ 
must appear as a submodule of $V$ since $V$ is holomorphic and $U$ and $M(0,0)$ are 
$C_2$-cofinite and rational. 
It implies $M(h_1, h_2)\neq 0$ for any $h_1, h_2\in \{0,\frac{1}2, \frac{1}{16}\}$.  
\end{proof}

The following two results follow immediately from the general
arguments on simple current extensions \cite{LY, Yamauchi2004}. 

\begin{lem}\label{cor:5.7}
Let $M= L(\frac{1}2, \frac{1}2)  \otimes L(\frac{1}2, \frac{1}2) \otimes M(0, 0)$ and  set
\[
\widetilde{M} = \bigoplus_{h_1, h_2\in \{0, \frac{1}2\}} L(\frac{1}2, \frac{1}{2}-h_1)  
\otimes L(\frac{1}2, \frac{1}2-h_2) \otimes M(h_1, h_2).
\]
Then $\widetilde{M}$ is an irreducible module of $V^E$.
\end{lem}

\begin{thm}
Let $t=\tau_e\tau_f$ and set
\[
X=\bigoplus_{i,j \in \{0,1\}} V^{(i,j)} \boxtimes_{V^E} \widetilde{M}.
\]
Then $X$ is an irreducible $t$-twisted module of $V$.
\end{thm}

Define 
\[
\widetilde{V}= V^{\langle t\rangle}\oplus X^{\langle t\rangle},
\]
where $X^{\langle t\rangle}$ is the irreducible $V^{\langle t\rangle}$-submodule of $X$ which has integral weights.  
Then $\widetilde{V}$ is a simple, rational, $C_2$-cofinite, holomorphic VOA of CFT-type. 
Notice that the conformal weight of 
$M= L(\frac{1}2, \frac{1}2)  \otimes L(\frac{1}2, \frac{1}2) \otimes M(0, 0)$  is $1$ 
and $M$ is an $L(\frac{1}2, 0)  \otimes L(\frac{1}2, 0) \otimes M(0, 0)$-submodule of 
$X^{\langle t \rangle}$. 
Hence, $(X^{\la t\ra})_1\neq 0$ and $(\widetilde{V})_1\neq 0$. 
Since $V_1=0$, we have $\tilde{V}_1= (X^{\la t\ra})_1$ 
and hence the Lie algebra on $\tilde{V}_1$ is abelian. Thus we have the following theorem.

\begin{thm}\label{thm:A7}
The VOA $\widetilde{V}$ is isomorphic to the Leech lattice VOA $V_\Lambda$.
\end{thm}

Now we are ready to prove our main theorem. 

\begin{thm}\label{2ising}
The VOA $V$ is isomorphic to the Moonshine VOA $V^\natural$.
\end{thm}

\begin{proof}
By Theorem \ref{thm:A7}, we know that the VOA $\widetilde{V}$ is
isomorphic to the Leech lattice VOA $V_\Lambda$.
Let $g$ be the automorphism of $\widetilde{V}$ which acts as $1$ on $V^{\la t\ra}$ and 
$-1$ on $X^{\la t\ra}$. 
Then $g$ is conjugate
to the lift $\theta$ of $-1$ map on $\Lambda$ since $g$ acts on $\widetilde{V}_1$ as $-1$ 
(cf. Theorem \ref{thm:v3}). 
Therefore, we have $\widetilde{V}^{\la g\ra}=V^{\la t\ra}\cong V_\Lambda^+$.
Then by the same argument as in Theorem \ref{thm:main}, we have 
\[
V\cong V_\Lambda^+ \oplus V_\Lambda^{T,+}\cong V^\natural
\]
as $V_\Lambda^+$-modules.
Then by the uniqueness of simple current extensions, 
we can establish the desired isomorphism between $V$ and $V^\natural$.
\end{proof}

\begin{rmk}
Recall that the Leech lattice $\Lambda$ contains a sublattice isometric to 
$\sqrt{2}E_8^{\oplus 3}$.
For $p=3, 5$, we can choose a fixed-point-free isometry $\tau$ of order $p$ such that 
each direct summand of $\sqrt{2}E_8^{\oplus 3}$ is stabilized; 
indeed, $\sqrt{2}E_8^{\oplus 3}$ contains $\sqrt{2}A_2^{\oplus 12}$ and 
$\sqrt{2}A_4^{\oplus 6}$ as sublattices and the fixed-point-free isometry of $\Lambda$ of 
order $3$ (resp. $5$) can be induced by the Coxeter element of $A_2$ (resp. $A_4$).  
Thus, we have $(V_{\sqrt2E_8}^{\la \tau\ra} )^{\otimes 3} \subset  V_{\Lambda}^{\la \tau\ra }$. 
Let $\theta\in\Aut V_{\sqrt2E_8}$ be a lift of the $-1$-isometry of $\sqrt2E_8$.
Then $\theta$ and $\tau$ commutes.
Since $V_{\sqrt2E_8}^{\la \theta\ra} $ has exactly $496$ Ising vectors \cite[Proposition 4.3]{LSY} 
and $496$ is relatively prime to $p$, there exists an Ising vector in $V_{\sqrt2E_8}^{\la \theta\ra}$ 
fixed by $\tau$.
Hence $V_{\Lambda}^{\la \tau\ra}$ contains two (in fact, three) mutually orthogonal Ising vectors. 
By Theorem \ref{2ising}, we have $\widetilde{V}_{\Lambda, \tau} \cong V^\natural$, also.
\end{rmk}

\paragraph{\textbf{Acknowledgment}} The authors thank Kenichiro Tanabe and Hiroshi Yamauchi for 
stimulating and valuable discussions and Masaaki Kitazume and Naoki Chigira for 
consultations about the Leech lattice.  
They also thank Scott Carnahan for pointing out a mistake in the early version.


\begin{thebibliography}{99}
\bibitem{BK2004}
B. Bakalov and V. G. Kac, 
Twisted modules over lattice vertex algebras
in \textit{Lie theory and its applications in physics} V, 3--26, World Sci.
Publ., River Edge, NJ, 2004.

\bibitem{CM2016}
S. Carnahan and M. Miyamoto, 
Regularity of fixed-point vertex operator subalgebras, arXiv:1603.05645.

\bibitem{CLS2017}
H.Y. Chen, C.H. Lam and H. Shimakura, 
$\Z_3$-orbifold construction of the Moonshine vertex
operator algebra and some maximal $3$-local subgroups of the Monster,
\textit{Math. Z.} DOI 10.1007/s00209-017-1878-z

\bibitem{ATLAS}
J. H. Conway, R. T. Curtis, S. P. Norton, R. A. Parker and R. A. Wilson, 
ATLAS of Finite Groups, Clarendon Press, Oxford, 1985.

\bibitem{DJX}
C. Dong, X. Jiao and F. Xu, Quantum dimensions and quantum Galois theory, 
\textit{Trans. Amer. Math. Soc.} \textbf{365} (2013), no.~12, 6441--6469.

\bibitem{DL1996} 
C. Dong and J. Lepowsky, 
The algebraic structure of relative twisted vertex operators, 
\textit{J. Pure, Appl. Math.} \textbf{110} (1996), 259--295. 

\bibitem{DLM2000}
C. Dong, H. Li and G. Mason, Modular-invariance of trace functions in
orbifold theory and generalized Moonshine, 
\textit{Commun. Math. Phys.} \textbf{214} (2000), 1--56.

\bibitem{DMZp}
C. Dong and G. Mason, The construction of the moonshine module as a $\Z_p$-orbifold,  
Mathematical aspects of conformal and topological field theories and quantum groups 
(South Hadley, MA, 1992), 37--52, \textit{Contemp. Math.}, \textbf{175}, 
Amer. Math. Soc., Providence, RI, 1994. 

\bibitem{DMq}
C. Dong and G. Mason, On quantum Galois theory, 
\textit{Duke Math. J.} \textbf{\bf 86} (1997), 305--321.

\bibitem{DM2004}
C. Dong and G. Mason, 
Rational vertex operator algebras and the effective central charge,
\textit{Int. Math. Res. Notices} 2004, No.56, 2989--3008.

\bibitem{DMb}
C. Dong and G. Mason, Holomorphic vertex operator algebras of small central charge, 
\textit{Pacific J. Math.} \textbf{213} (2004), 253--266.

\bibitem{DN1999}
C. Dong and K. Nagatomo, 
Automorphism groups and twisted modules for lattice vertex operator algebras, 
\textit{Contemp. Math.} \textbf{248} (1999), 117--133.

\bibitem{DRX2015}
C. Dong, L. Ren and F. Xu, 
On orbifold theory, arXiv:1507.03306v2.

\bibitem{EMS2015}
J. van Ekeren, S. M\"{o}ller and N. Scheithauer, Construction and classification of holomorphic
vertex operator algebras, arXiv:1507.08142v2.

\bibitem{FLM1988}
I. Frenkel, J. Lepowsky and A. Meurman, 
\textit{Vertex operator algebras and the Monster}, 
Pure Appl. Math. \textbf{134}, 
Academic Press, Boston, MA,1988.

\bibitem{GL2011}
R. L. Griess, Jr. and C.H. Lam, 
A moonshine path for $5A$ and associated lattices of ranks $8$ and $16$, 
\textit{J. Algebra} \textbf{331} (2011), 338--361.

\bibitem{KMi}
M. Krauel and M. Miyamoto, 
A modular invariance property of multivariable trace functions for regular vertex operator algebras, 
\textit{J. Algebra} \textbf{444} (2015), 124--142.

\bibitem{LSY} C.H. Lam, S. Sakuma and H. Yamauchi, 
Ising vectors and automorphism groups of commutant subalgebras related to root systems, 
\textit{Math. Z.} \textbf{255} (2007), no. 3, 597--626. 

\bibitem{LY1}
C.H. Lam and H. Yamauchi,
A characterization of the moonshine vertex operator algebra by means of Virasoro frames, 
\textit{Intern. Math. Res. Notices}, 2007 (2007).

\bibitem{LY}
C.H. Lam and H. Yamauchi, On the structure of framed vertex operator algebras
and  their pointwise frame stabilizers,  \textit{Comm. Math. Phys.} \textbf{277}
(2008),  237--285.

\bibitem{Lepowsky1985} 
J. Lepowsky,  
Calculus of twisted vertex operators, 
\textit{Proc. Natl. Acad. Sci. USA} \textbf{82} (1985), 8295--8299. 

\bibitem{M1}
M. Miyamoto, Griess algebras and conformal vectors in vertex operator algebras, 
\textit{J. Algebra} \textbf{179} (1996), 528--548.

\bibitem{Miyamoto2013}
M. Miyamoto, A $\Z_3$-orbifold theory of lattice vertex operator algebra and 
$\Z_3$-orbifold constructions, 
in \textit{Symmetries, Integrable Systems and Representations}, 
Springer Proceedings in Mathematics and Statistics,
\textbf{40}, 319--344, Springer, Heidelberg, 2013.

\bibitem{Miyamoto2015}
M. Miyamoto, $C_2$-cofiniteness of cyclic-orbifold models, 
\textit{Commun .Math. Phys.} \textbf{335} (2015), 1279--1286.

\bibitem{MT2004}
M. Miyamoto and K. Tanabe, 
Uniform product of $A_{g,n}(V)$ for an orbifold model $V$ and $G$-twisted Zhu algebra, 
\textit{J. Algebra} \textbf{274} (2004), 80--96.

\bibitem{Moeller2016}
S. M\"{o}ller, 
A cyclic orbifold theory for holomorphic vertex operator algebras and applications, 
arXiv:1611.09843.

\bibitem{Sh07}
H. Shimakura, Lifts of automorphisms of vertex operator algebras in
  simple current extensions, \textit{Math. Z.} \textbf{256} (2007), no.~3, 491--508.

\bibitem{TY2013}
K. Tanabe and H.  Yamada, Fixed point subalgebras of lattice vertex operator algebras 
by an automorphismof order three. 
\textit{J. Math. Soc. Japan} \textbf{65} (2013), 1169--1242.

\bibitem{Yamauchi2004}
H. Yamauchi, 
Module categories of simple current extensions of vertex operator algebras, 
\textit{J. Pure Appl. Algebra} \textbf{189} (2004), 315--328.
\end{thebibliography}
\end{document}